\documentclass[12pt,reqno]{amsart}
\usepackage{hyperref}
\usepackage{amsmath,amssymb, amsthm}
\usepackage{caption}
\usepackage{amssymb,amsmath,euscript,enumerate,tikz}
\usepackage[margin=1in]{geometry}
\usepackage{graphicx}
\usepackage{float}
\usepackage{pgf,tikz}
\usepackage{mathrsfs}
\usepackage{mathtools}

\newcommand \reg{\operatorname{reg}}

\newcommand\supp{\operatorname{supp}}
\newcommand \Tor{\operatorname{Tor}}

\newcommand{\KK}{\mathbb{K}}

\theoremstyle{plain}
\newtheorem{theorem}{Theorem}[section]
\newtheorem{lemma}[theorem]{Lemma}
\newtheorem{proposition}[theorem]{Proposition}
\newtheorem{corollary}[theorem]{Corollary}

\theoremstyle{definition}

\newtheorem{notation}[theorem]{Notation}
\newtheorem{remark}[theorem]{Remark}
\newtheorem{definition}[theorem]{Definition}
\newtheorem{example}[theorem]{Example}

\newtheorem{problem}[theorem]{Problem}

\begin{document}
\title[Regularity of $3$-Path Ideals of Graphs]{Regularity of $3$-Path Ideals of Trees and Unicyclic Graphs}
\author[Rajiv Kumar]{Rajiv Kumar}
\email{gargrajiv00@gmail.com, rajiv.kumar@iitjammu.ac.in }
\address{Department of Mathematics, 
	Indian Institute of Technology Jammu,\newline \indent
	Jammu, JK, 181221, India.}
\email{rajib.sarkar63@gmail.com}
\author[Rajib Sarkar]{Rajib Sarkar}
\address{Stat-Math Unit, Indian Statistical Institute \newline \indent Kolkata, 203 B.T. Road, Kolkata--700108, India.}

	\begin{abstract}
 Let $G$ be a simple graph and $I_3(G)$ be its $3$-path ideal in the corresponding polynomial ring $R$. In this article, we prove that for an arbitrary graph $G$, $\reg(R/I_3(G))$ is bounded below by $2\nu_3(G)$, where $\nu_3(G)$ denotes the $3$-path induced matching number of $G$. We give a class of graphs, namely, trees for which the lower bound is attained. Also, for a unicyclic graph $G$, we show that $\reg(R/I_3(G))\leq 2\nu_3(G)+2$ and provide an example that shows that the given upper bound is sharp.
	\end{abstract}
 \keywords{$t$-path ideal, trees and unicyclic graphs, Castelnuovo-Mumford regularity}
\thanks{AMS Subject Classification (2020): 13D02, 13C13, 05E40}
	
	\maketitle
	\section{Introduction}
	Let $G$ be a finite simple graph with the vertex set $V(G)$ and the edge set $E(G)$. Let $R=\mathbb{K}[x:x\in V(G)]$ be the polynomial ring, where $\mathbb{K}$ is an arbitrary field. In \cite{Conca-Negri-1999}, the notion of path ideals has been introduced. For a graph $G$, a square-free monomial ideal,
	$$I_t(G):=\langle x_1\cdots x_t: \text{ where } P:x_1,\dots,x_t \text{ is a } t\text{-path in } G\rangle \subseteq R$$
 is called \textit{$t$-path ideal} of $G$. For the last few decades, researchers have been trying to establish the connection between the combinatorial invariants of graphs and the algebraic invariants of $t$-path ideal, see \cite{AF18,AF15,BBT17,BHO11,Erey20,KM16,KO14}. The Castelnuovo-Mumford regularity of an ideal is an important algebraic invariant that measures the complexity of the module. The \textit{Castelnuovo–Mumford
regularity} (or simply regularity) of a finitely generated graded $R$-module $M$, written $\reg(M)$ is defined as
\( \reg(M):=\max \{j-i: \Tor_i(M,\KK)_j\neq 0\}.\)
There are a few classes of graphs for which the explicit formula of the regularity of $t$-path ideal is known. In \cite{BHO11}, Bouchat et al. studied the $t$-path ideal of rooted trees, and they provided a recursive formula for computing the graded Betti numbers of $t$-path ideals. Also, they gave a general bound for the regularity of $t$-path ideal of a rooted tree. In particular, for a rooted tree $G$, they proved that $\reg(R/I_t(G))\leq (t-1)[l_t(G)+p_t(G)],$ where $l_t(G)$ denotes the number of leaves in $G$, where level is at least $t-1$ and $p_t(G)$ denotes the maximal number of pairwise disjoint paths of length $t$ in $G$. In \cite{AF18} and \cite{AF15}, Alilooee and Faridi computed the regularity of $t$-path ideal of lines and cycles in terms of the number of vertices, respectively. Banerjee \cite{Banerjee17} studied the regularity of $t$-path ideal of gap free graphs and proved that the $t$-path ideals of gap free, claw free and
whiskered-$K_4$ free graphs have linear minimal free resolutions for all $t\geq 3$. In this article, we restrict ourselves to $t=3$ and study the regularity  of $I_3(G)$. It is important to note that for a $3$-path ideal of gap free graph $G$, Banerjee proved that $\reg(R/I_3(G))\leq \max \{\reg(R/I_2(G)),2\},$ where $I_2(G)$ is the $2$-path ideal or monomial edge ideal of $G$, see \cite[Theorem 3.3]{Banerjee17}. In \cite{Katzman06}, Katzman proved that for any graph $G$, $\reg(R/I_2(G))\geq \nu(G)$, where $\nu(G)$ denotes the induced matching number of $G$. Motivated from the definition of $\nu(G)$, in this article, we define in an obvious way the $3$-path induced matching number, denoted by $\nu_3(G)$ (see section \ref{sec:prel} for the definition). Next, we prove that a similar lower bound can be obtained for the regularity of $I_3(G)$. More precisely, we prove $\reg(R/I_3(G))\geq 2\nu_3(G)$. We also observe that this is a sharp lower bound for $\reg(R/I_3(G))$. In fact, we show that if $G$ is a tree, then $\reg(R/I_3(G))=2\nu_3(G)$, see Theorem \ref{thm:trees}.
 It is desirable to answer the following problem:
 \begin{problem}\label{problem} Classify the classes of graphs $G$ that satisfy the property $$\reg(R/I_3(G))=2\nu_3(G).$$
 \end{problem}
 Theorem \ref{thm:trees} gives a class of graphs that satisfies the property $\reg(R/I_3(G))=2\nu_3(G)$. On the other hand, Example \ref{example} shows that some of the unicyclic graphs satisfy the desired property but not the whole class. More concretely, in Theorem \ref{thm:unicyclic}, we prove that if $G$ is a unicyclic graph, then $\reg(R/I_3(G))\leq 2\nu_3(G)+2$. Moreover, in Example \ref{example}, we give examples of unicyclic graphs showing that the regularity of $3$-path ideal of a unicyclic graph can attain any of the values between the lower and upper bound of the regularity.

 \textbf{Acknowledgment:}
The second author thanks the National Board for Higher Mathematics (NBHM), Department of Atomic Energy, Government of India, for financial support through Postdoctoral Fellowship.
	\section{Preliminaries}\label{sec:prel}
	In this section, we recall all necessary definitions which will be used throughout the article.
	\begin{definition}
	\begin{enumerate}[i)]
	Let $G$ be a graph with the vertex set $V(G)=\{x_1,\dots,x_n\}$ and the edge set $E(G)$.
	   \item A subgraph $H$ of a graph $G$ is called an \textit{induced subgraph} if for all $x_i,x_j\in V(H)$ such that $\{x_i,x_j\}\in E(G)$ implies that $\{x_i,x_j\}\in E(H)$. For a vertex $x\in V(G)$, let $G\setminus \{x\}$ denote the induced subgraph on the vertex set $V(G)\setminus \{x\}.$

	   \item A \textit{path} on $n$ vertices is a graph whose vertices can be listed in the order $x_1,\dots,x_n$ such that the edge set is $\{\{x_i,x_{i+1}\}:1\leq i\leq n\}$ and it is denoted by $P_n$. For $t\geq 2$, a path of length $t$ in $G$ is called a \textit{$t$-path}. 
	    \item A \textit{cycle} on $n$ vertices $\{x_1,\dots,x_n\}$, denoted by $C_n$, is the graph with the edge set $E(P_n)\cup \{\{x_1,x_n\}\}.$ A graph $G$ is said to be \textit{tree} if it does not contain any cycle. A disconnected tree is called a \textit{forest}. A graph $G$ is called \textit{unicyclic} if $G$ contains only one cycle.
	    \item For a vertex $x\in V(G)$, the set $\{y\in V(G): \{x,y\}\in E(G)\}$ is called the \textit{neighborhood} of $x$ in $G$ and it is denoted by $N_G(x)$. The set $N_G[x]$ denotes $N_G(x)\cup \{x\}$. For an edge $e=\{x,y\}\in E(G)$, the \textit{neighborhood} of $e$ is defined as 
	    $$N_G(e):=(N_G(x)\setminus \{y\})\cup (N_G(y)\setminus \{x\}).$$ 
	    $N_G[e]$ denotes the set $N_G(e)\cup \{x,y\}$, i.e., $N_G[e]=N_G[x]\cup N_G[y].$
	    \item For a vertex $x\in V(G)$, the set $\{ \{y,z\}\in E(G): \{x,y,z\} \text{ is a } 3\text{-path in }G \}$ is called the \textit{neighborhood edge set} of $x$ in $G$ and it is denoted by $N_G^{\text edge}(x)$.
     \item A \textit{3-path matching} in a graph $G$ is a subgraph consisting pairwise disjoint $3$-paths. If the subgraph is induced, then $3$-path matching is said to be a \textit{$3$-path induced matching} of $G$. The largest size of a $3$-path induced matching is called the \textit{$3$-path induced matching number}, and it is denoted by $\nu_3(G)$.
	\end{enumerate}
	\end{definition}
	The following remark shows the relationship between the $3$-path induced matching number of a graph and its induced subgraph.
 \begin{remark}\label{rmk:subgraph-nu}
	 One can observe that for an induced subgraph $H$ of $G$, we have $\nu_3(H)\leq \nu_3(G)$.
	\end{remark}

	\begin{example}
	\noindent

	\begin{minipage}{\linewidth}
		\begin{minipage}{.5\linewidth}
		 Note that $\{x_1,x_2,x_3\},\{x_4,x_5,x_6\}$ is a $3$-path matching but not a $3$-path induced matching in $G$. Here, $\{x_1,x_2,x_7\},\{x_4,x_5,x_6\}$ is a $3$-path induced matching in $G$ and $\nu_3(G)=2$.
		\end{minipage}
		\begin{minipage}{.5\linewidth}
			\begin{figure}[H]
				\begin{tikzpicture}[scale=1.3]
\draw (-1,3.56)-- (-1,2.7);
\draw (-1,2.7)-- (-2,2.7);
\draw (-1,2.7)-- (0.0,2.7);
\draw (0.0,2.7)-- (1,2.7);
\draw (2,2.7)-- (1,2.7);
\draw (2,2.7)-- (3,2.7);
\begin{scriptsize}
\fill (-1,3.56) circle (1.5pt);
\draw (-1,3.8) node {$x_7$};
\fill (2,2.7) circle (1.5pt);
\draw  (2,2.9) node {$x_5$};
\fill (3,2.7) circle (1.5pt);
\draw  (3, 2.9) node {$x_6$};
\fill (-2,2.7) circle (1.5pt);
\draw (-2,2.9) node {$x_1$};
\fill (-1,2.7) circle (1.5pt);
\draw (-1.2,2.9) node {$x_2$};
\fill (0.0,2.7) circle (1.5pt);
\draw (0.0,2.94) node {$x_3$};
\fill (1,2.7) circle (1.5pt);
\draw (1,2.9) node {$x_4$};
\end{scriptsize}
\end{tikzpicture}
				\caption{$G$}
			\end{figure}
		\end{minipage}
	\end{minipage}
	\end{example}

	\section{A lower bound for the regularity}
For an arbitrary graph $G$, we first give a general lower bound for the regularity of $3$-path ideal in terms of the regularity of $3$-path ideal of its induced subgraph, and as a consequence, we prove that the regularity of $R/I_3(G)$ is bounded below by $2\nu_3(G)$. 
\begin{proposition}\label{prop:induced-subgraph}
		Let $H$ be an induced subgraph of $G$. Then $$\beta_{i,j}(R_H/I_3(H))\leq \beta_{i,j}(R/I_3(G))$$
		where $R_H=\mathbb{K}[x:x\in V(H)]$. Moreover, $\reg(R_H/I_3(H))\leq \reg(R/I_3(G)).$
	\end{proposition}
\begin{proof}
We prove that $R_H/I_3(H)$ is an algebra retract of $R/I_3(G)$. We first show that $I_3(H)=I_3(G)\cap R_H$. It is clear that $I_3(H) \subseteq I_3(G)\cap R_H$. For the converse part, let $f\in I_3(G)\cap R_H$. Suppose $f=\sum gh$ for $g\in R$ and $h\in I_3(G)$. We consider the mapping $\varphi:R\longrightarrow R_H$ by defining $\varphi(x)=x$ if $x\in V(H)$, otherwise $\varphi(x)=0$. Therefore,
	\begin{align*}
	\varphi(f) & = \sum \varphi(g)\varphi(h) \\
	&=\sum \varphi(g)h, \text{ where } h\in I_3(H).
	\end{align*}
	This implies that $f\in I_3(H)$, and so $I_3(H)=I_3(G)\cap R_H$. Thus, $R_H/I_3(H)$ is a $\mathbb{K}$-subalgebra of $R/I_3(G)$. Now consider $R_H/I_3(H) \xhookrightarrow{i}  R/I_3(G) \stackrel{\bar{\varphi}}\rightarrow R_H/I_3(H)$, where $\bar{\varphi}$ is the map induced by $\varphi$. It can be observed that $\bar{\varphi}\circ i$ is the identity map on $R_H/I_3(H)$. Hence, $R_H/I_3(H)$ is an algebra retract of $R/I_3(G)$.
	Now, the assertion follows
	from \cite[Corollary 2.5]{HHO}.
	\end{proof}
 As an immediate consequence, we have a lower bound of the regularity of $3$-path ideal of any graph $G$ in terms of the combinatorial invariants of $G$.
\begin{corollary}\label{cor:lower-bound}
Let G be a simple graph and $I_3(G)$ be its 3-path ideal. Then $$\reg(R/I_3(G))\geq 2\nu_3(G).$$
	\end{corollary}
\begin{proof}
	For simplicity of notation, let $s=\nu_3(G)$. Suppose that $\{P_1,P_2,\dots,P_{s}\}$ is a $3$-path induced matching in $G$. Let $H$ be the induced subgraph of $G$ on the vertices $\cup_{i=1}^{s} V(P_i)$. Then $I_3(H)$ is a complete intersection. Thus by the Koszul complex, $\reg(R_H/I_3(H))=2\nu_3(G)$. Hence, the assertion follows from Proposition \ref{prop:induced-subgraph}.
	\end{proof}

	\section{Path Ideals of trees and Unicyclic Graphs}
	In this section, we consider the $3$-path ideal of trees and unicyclic graphs. In fact, we compute the exact regularity of $I_3(G)$ when $G$ is a tree, and for unicyclic graphs, we give a sharp upper bound of $\reg(I_3(G))$. We first prove some technical lemmas which will be needed to prove the main results.
	\begin{lemma}\label{lemm:tech-lemma}
	    Let $G$ be a simple graph and $I_3(G)$ be its $3$-path ideal. Let $e=\{x,y\}\in E(G)$. Then we have the followings:
	    \begin{enumerate}
	        \item $I_3(G):xy=L+J$, where $L=\left\langle N_G(e)\right\rangle$  and $J=I_3(G\setminus N_G[e])$.
	        \item $(I_3(G)+\langle xy \rangle):x= \langle y \rangle +I_2(H)+I_3(G\setminus N_G[x]),$ where $H$ is the union of $N_G^{\text edge}(x)$ and the complete graph on the vertex set $N_G(x)\setminus \{y\}$.
	    \end{enumerate}
	\end{lemma}
	\begin{proof}
	    (1): Clearly $L+J\subset I_3(G):xy$. On the other side, let $u\in I_3(G):xy$. This implies that $uxy\in I_3(G)$, and hence there exists a minimal monomial generator $v\in I_3(G)$ such that $v\mid uxy$. If $v\nmid u$, then $\gcd(v, xy)\neq 1$. This forces that $\supp(u)\cap N_G(e)\neq \emptyset$. This gives $u\in L$, and hence $I_3(G):xy=L+J$. 
	        
	        (2): It can be easily seen that  $\langle y \rangle +I_2(H)+I_3(G\setminus N_G[x])\subset (I_3(G)+\langle xy \rangle):x$. For the converse part, let $u \in (I_3(G)+\langle xy \rangle):x$. Then $ux\in I_3(G)+\langle xy \rangle$. Assume that $y\nmid u$. Then $ux\in I_3(G)$, and hence there exists a minimal monomial generator $v\in I_3(G)$ such that $v\mid ux$. If $v\nmid u$, then $x\mid v$. Let $v=xv_1$. Since $v$ is $3$-path in $G$, hence $x\in N_G(v_1)$. Now, if $v_1$ is an edge in $G$, then $v_1\in N^\text{edge}_G(x)$, on the other hand $\supp(v_1)\subset N_G(x)$ which further implies that $u\in I_2(H)$. This yields that $(I_3(G)+\langle xy \rangle):x\subset \langle y \rangle +I_2(H)+I_3(G\setminus N_G[x])$, and hence $$(I_3(G)+\langle xy \rangle):x= \langle y \rangle +I_2(H)+I_3(G\setminus N_G[x])$$
 which completes the proof.
 \end{proof}
	
	We recall a property satisfied by a tree from \cite[Proposition 4.1]{Jacques-Katzman}.
\begin{remark}\label{tree-property}
	If $G$ is a tree containing a vertex of degree at least two, then there exists a vertex $v\in V(G)$ with $N_G(v)=\{v_1,\dots,v_r\}$, where $r\geq 2$ and $\deg_G(v_i)=1$ for $i< r$.
\end{remark}
	

\begin{lemma}\label{lemm:tech-lemma2}
    Let $G$ be a tree and $v\in V(G)$ with the notation as in Remark \ref{tree-property}. Then 
    $$\nu_3(G\setminus N_G[e])\leq \nu_3(G)-1,\text{ where } e=\{v,v_r\}$$
\end{lemma}
\begin{proof}
    Let $\nu_3(G\setminus N_G[e])=s$ and $P_1,\dots,P_s$ be a $3$-path induced matching in $G\setminus N_G[e]$. Since $\deg_G(v_1)=1$, $N_G(v_1)\cap V(P_i)=\emptyset$ for $1\leq i\leq s$. Also, note that for $1\leq i\leq s$, $N_G(v)\cap V(P_i)=\emptyset=N_G(v_r)\cap V(P_i)$. Therefore, $\{v_1,v,v_r\},P_1,\dots,P_s$ is a $3$-path induced matching in $G$, and hence $\nu_3(G)\geq s+1$.
\end{proof}
Now we are ready to prove the exact regularity formula for the $3$-path ideals of trees.
	\begin{theorem}\label{thm:trees}
		Let G be a tree and $I_3(G)$ be its 3-path ideal. Then $\reg(R/I_3(G))=2\nu_3(G)$.
	\end{theorem}
\begin{proof}
	By Corollary \ref{cor:lower-bound}, it is enough to prove that $\reg(R/I_3(G))\leq 2\nu_3(G)$.
	
	We proceed by induction on $|V(G)|$. By Remark \ref{tree-property}, let $v$ be such a vertex of $G$ such that $N_G(v)=\{v_1,\dots,v_r\}$ and $\deg_G(v_i)=1$ for $i<r$, where $r\geq 2$. Set $u=v_r$, and let $N_G(u)=\{u_1=v,u_2,\dots,u_s\}$ for $s\geq 1$. Consider the following short exact sequence:
	\begin{align}\label{eqn-1}
	0\longrightarrow \frac{R}{I_3(G):uv}(-2)\stackrel{\cdot uv}{\longrightarrow} \frac{R}{I_3(G)}\longrightarrow \frac{R}{\langle uv,I_3(G)\rangle }\longrightarrow 0.
\end{align}
	Using Lemma \ref{lemm:tech-lemma}(1), we have $$I_3(G):uv=\langle v_i:1\leq i\leq r-1\rangle +\langle u_j:2\leq j\leq s\rangle +I_3(G\setminus N_G[e]), \text{ where } e=\{u,v\}.$$ By virtue of Lemma \ref{lemm:tech-lemma2}, we have $\nu_3(G\setminus N_G[e]) \leq \nu_3(G)-1$.
	Therefore, by inductive hypothesis, $\reg(R/(I_3(G):uv))\leq 2\nu_3(G\setminus N_G[e])\leq 2\nu_3(G)-2.$ Now, set $J=\langle uv,I_3(G)\rangle$ and consider the following short exact sequence:
	\begin{align}\label{eqn-2}
	0\longrightarrow \frac{R}{J:u}(-1)\stackrel{\cdot u}{\longrightarrow} \frac{R}{J}\longrightarrow \frac{R}{\langle u, J\rangle}\longrightarrow 0.
	\end{align}
	Observe that, $\langle u, J\rangle=\langle u \rangle +I_3(G\setminus\{u\})$. By Remark \ref{rmk:subgraph-nu}, $\nu_3(G\setminus \{u\})\leq \nu_3(G).$ Thus it follows from  inductive hypothesis that $\reg(R/\langle u, J\rangle)\leq 2\nu_3(G\setminus \{u\})\leq 2\nu_3(G).$ On the other hand,
	by Lemma \ref{lemm:tech-lemma}(2), $J:u=\langle v\rangle +I_2(H)+I_3(G\setminus N_G[u]),$ where $H$ is the union of $N_G^{\text edge}(u)$ and the complete graph on the vertex set $N_G(u)\setminus \{v \}$.\\
	\emph{Case-I:} If $s=1$, then $N_G^{\text edge}(u)=\{\{v,v_i\}:1\leq i\leq r-1 \}$. Also, note that the graph $G\setminus N_G[u]$ does not have an edge. This implies that $J:u=\langle v\rangle $, and hence $\reg\left(S/(J:u)\right)=0\leq 2\nu_3(G)-1$.\\
	\emph{Case-II:} Suppose $s\geq 2$ and $\deg(u_i)=1$ for $2\leq i\leq s$. Then  $H$ is the union of $N_G^{\text edge}(u)=\{\{v,v_i\}:1\leq i\leq r-1 \}$ and the complete graph on the vertex set $\{u_2,\dots, u_s\}$. In this case, the graph $G\setminus N_G[u]$ has no edge. Thus, $J:u=\langle v\rangle +I_2(H)$. Since $H$ is co-chordal, by \cite[Theorem 1]{Frob90}, $I_2(H)$ has a linear resolution. Hence, we get $\reg(R/(J:u))=1\leq 2\nu_3(G)-1$.

	Therefore, it follows from \cite[Corollary 18.7]{Peeva} applying to the short exact sequences \eqref{eqn-1} and \eqref{eqn-2} that 
	\[ \reg(R/I_3(G))\leq \max \{ \reg(R/(I_3(G):uv))+2,\reg(R/(J:u))+1,\reg(R/\langle u,J\rangle )\}.
	\]
	Hence, $\reg(R/J)\leq 2\nu_3(G)$.\\
	\emph{Case-III:} Suppose now $s\geq 2$ and $\deg_G(u_i)\geq 2$ for some $2\leq i\leq s$. Without loss of generality, we assume that $\deg_G(u_i)\geq 2$ for all $2\leq i\leq t$, and $\deg(u_i)=1$ for all $t+1\leq i\leq s$.
	In this case, $H$ is a union of the edges $N_G^{\text{edge}}(u)$ and the complete graph on the vertex set $\{u_2,\dots, u_s\}$, i.e., $$J:u=\langle v\rangle +\langle u_iu_j: 2\leq i< j\leq s\rangle +\sum\limits_{i=2}^t\left\langle u_ix :  x\in N_G(u_i)\setminus\{u\}\right\rangle+I_3\left(G\setminus N_G[u]\right). $$ 
Set $J_1=J:u$ and $J_{i}=J_{1}+\langle u_2,\dots,u_{i}\rangle $ for $2\leq i\leq t$. For $2\leq i\leq t$, consider the following short exact sequence:
	\begin{align}\label{eqn-3}
	0\longrightarrow \frac{R}{J_{i-1}:u_i}(-1)\stackrel{\cdot u_i}{\longrightarrow} \frac{R}{J_{i-1}}\longrightarrow \frac{R}{J_{i}}\longrightarrow 0.
	\end{align}
	
	By Lemma \ref{lemm:tech-lemma}, $J:uu_i=\langle u_j:1\leq j\leq s \text{ and } j\neq i\rangle+ \langle w: w\in N_G(u_i)\setminus \{u\}\rangle +I_3(G\setminus N_G[\{u,u_i\}]$. Therefore, 
 \begin{align*}
     J_{i-1}:u_i & =
 (J_{1}+\langle u_2,\dots,u_{i-1}\rangle ):u_i=
 (J:uu_i)+\langle u_2,\dots,u_{i-1}\rangle \\ 
 &= \langle u_j:1\leq j\leq s, \ j\neq i\rangle+\langle w: w\in N_G(u_i)\setminus \{u\}\rangle +I_3(G\setminus N_G[\{u,u_i\}].
 \end{align*}
 Further, we can write
 $$J_{i-1}:u_i= \langle u_j:1\leq j\leq s, \ j\neq i\rangle +\langle w: w\in N_G(u_i)\setminus \{u\}\rangle +I_3(G\setminus \{N_G[v]\cup N_G[u]\cup N_G[u_i]\})$$
 for $2\leq i\leq t$. Also, $J_{t}=\langle v,u_2,\dots,u_t\rangle +I_3(G\setminus \{N_G[v]\cup N_G[u]\})$. Now it follows from \cite[Corollary 18.7]{Peeva} applying to the short exact sequence \eqref{eqn-3} that 
	\begin{align*}
	\reg(R/J_1)\leq \max \{\reg(R/(J_{i-1}:u_i))+1, \reg(R/J_{t}):2\leq i\leq t \}.
	\end{align*}
	Since $G\setminus \{N_G[v]\cup N_G[u]\cup N_G[u_i]\}$ is an induced subgraph of $G\setminus N_G[e]$, by Remark \ref{rmk:subgraph-nu} and Lemma \ref{lemm:tech-lemma2}, we have $\nu_3(G\setminus \{N_G[v]\cup N_G[u]\cup N_G[u_i]\})\leq \nu_3(G\setminus \{N_G[v]\cup N_G[u]\})\leq \nu_3(G)-1$.
	
	By inductive hypothesis, $\reg(R/(J_{i-1}:u_i))\leq 2\nu_3(G\setminus \{N_G[v]\cup N_G[u]\cup N_G[u_i]\})\leq 2(\nu_3(G)-1)$ and 
	$\reg(R/J_{t})\leq 2\nu_3(G\setminus \{N_G[v]\cup N_G[u]\})\leq 2\nu_3(G)-2$.
By applying \cite[Corollary 18.7]{Peeva} to the above short exact sequences, we have
	\begin{align*}
	& \reg(R/I_3(G))  \leq \max \{\reg(R/(I_3(G):uv))+2,\reg(R/J) \} \\
	& \ \ \ \ \ \ \ \ \ \ \ \ \ \ \ \ \ \ \leq \max \{\reg(R/(I_3(G):uv))+2,\reg(R/\langle u,J\rangle ),\reg(R/J_1)+1 \} \\
	& \ \ \ \ \ \ \ \ \ \ \leq \max \{\reg(R/(I_3(G):uv))+2,\reg(R/\langle u,J\rangle ),\reg(R/(J_1:u_2))+2,\reg(R/J_2)+1 \} \\
	& \ \ \ \ \ \leq \max_{2\leq i\leq t} \{\reg(R/(I_3(G):uv))+2,\reg(R/\langle u,J\rangle ),\reg(R/(J_{i-1}:u_i))+2,\reg(R/J_{t})+1\}.
	\end{align*}
	Hence, the assertion follows.
	
	\end{proof}
 
Now we proceed to study the regularity of $3$-path ideal of unicyclic graphs. If $G$ is a cycle, then the regularity of $R/I_3(G)$ has been computed in \cite{AF15}. So, we assume that $G$ is not a cycle. We give a sharp upper bound for the regularity of $R/I_3(G)$. The idea of the proof is kind of similar to the proof of Theorem \ref{thm:trees}. We fix the following notation for unicyclic graphs.
\begin{notation}\label{notation-unicyclic}
    Let $G$ be a unicyclic graph with the induced cycle $C$. Then trees are attached to at least one vertex of $C$, say $u\in V(C)$. Let $v\in N_G(u)\setminus V(C)$ and $e=\{u,v\}$. Clearly $N_G(u)\setminus \{v\}$ contains at least $2$ vertices and set $N_G(u)=\{u_1=v,u_2,\dots,u_t\}$ for $t\ge 3$.
\end{notation}

  
\begin{theorem}\label{thm:unicyclic}
	Let G be a unicyclic graph and $I_3(G)$ be its 3-path ideal. Then $$2\nu_3(G)\leq \reg(R/I_3(G))\leq 2\nu_3(G)+2.$$
\end{theorem}
\begin{proof}
	Let $G$ be a unicyclic graph with the notation as in Notation \ref{notation-unicyclic}. The lower bound for $\reg(R/I_3(G))$ follows from Corollary \ref{cor:lower-bound}. So, here we only establish the upper bound. Consider the short exact sequence \eqref{eqn-1}.
 By Lemma \ref{lemm:tech-lemma}(1), $I_3(G):uv=\langle N_G(e)\rangle +I_3(G\setminus N_G[e])$, where $e=\{u,v\}$. Since $G\setminus N_G[e]$ is an induced subgraph of $G$, by Remark \ref{rmk:subgraph-nu}, $\nu_3(G\setminus N_G[e])\leq \nu_3(G)$. Note that $G\setminus N_G[e]$ is a tree. Thus, it follows from Theorem \ref{thm:trees} that $$\reg(R/(I_3(G):uv))=2\nu_3(G\setminus N_G[e])\leq 2\nu_3(G).$$
	Now set $J=\langle uv,I_3(G)\rangle$ and we consider the short exact sequence \eqref{eqn-2}, where $J:u=\langle v\rangle +I_2(H) +I_3(G\setminus \{N_G[u]\}),$ where $H$ is the union of $N_G^{\text edge}(u)$ and the complete graph on the vertex set $N_G(u)\setminus \{v\}$. Also, $\langle u,J\rangle =\langle u,I_3(G\setminus \{u\})\rangle $. Since $G\setminus \{u\}$ is a tree, by Theorem \ref{thm:trees}, we have
	$$\reg(R/\langle u,J\rangle)=2\nu_3(G\setminus \{u\})\leq 2\nu_3(G).$$
	Set $J_1=J:u$ and $J_i=J_{1}+\langle  u_2,\dots,u_{i}\rangle $, where $N_G(u)=\{v,u_2,\dots,u_t\}$ for $2\leq i\leq t$ and consider short exact sequences \eqref{eqn-3}.
	
	It can be observed that $J_{i-1}:u_i=\langle u_j:1\leq j\leq t, \ j\neq i\rangle +\langle w: w\in N_G(u_i)\setminus \{u\}\rangle +I_3(G\setminus \{N_G[u]\cup N_G[u_i]\})$ for $2\leq i\leq t$ and $J_{t}=\langle v,u_2,\dots,u_t\rangle +I_3(G\setminus \{N_G[u]\})$. Now it follows from \cite[Corollary 18.7]{Peeva} applying to the short exact sequence \eqref{eqn-3} that 
	\begin{align*}
	\reg(R/J_1)\leq \max \{\reg(R/(J_{i-1}:u_i))+1, \reg(R/J_{t}):2\leq i\leq t \}.
	\end{align*}
	Now, $\nu_3(G\setminus \{N_G[u]\cup N_G[w_i]\})\leq \nu_3(G)$ and $\nu_3(G\setminus \{N_G[u]\})\leq \nu_3(G)$ follow from Remark \ref{rmk:subgraph-nu}. Since $G\setminus \{N_G[u]\cup N_G[w_i]\}$ and $G\setminus \{N_G[u]\}$ are trees, by Theorem \ref{thm:trees}, $\reg(R/(J_i:w_i))=2\nu_3(G\setminus \{N_G[u]\cup N_G[w_i]\})\leq 2\nu_3(G)$ and 
	$\reg(R/J_{t})=2\nu_3(G\setminus \{N_G[u]\})\leq 2\nu_3(G)$. Therefore, it follows from applying \cite[Corollary 18.7]{Peeva} to short exact sequences \ref{eqn-1}, \ref{eqn-2} and \ref{eqn-3} that
	$\reg(R/I_3(G))\leq 2\nu_3(G)+2$.

\end{proof}

We now show by examples that all the three possibilities for the regularity of $R/I_3(G)$, namely, $2\nu_3(G)$, $2\nu_3(G)+1$, and $2\nu_3(G)+2$ indeed occur for unicyclic graphs.
	\begin{example}\label{example}
      Consider graphs $G_1$, $G_2$ and $G_3$ as in Figure \ref{fig:unicyclic}. Then using Macaulay 2 (\cite{M2}), it can be computed that $\reg(R/I_3(G_1))=2$, $\reg(R/I_3(G_2))=3$, and $\reg(R/I_3(G_3))=6$. Note that $\nu_3(G_1)=1=\nu_3(G_2)$ and $\nu_3(G_3)=2$.
	
	\noindent
	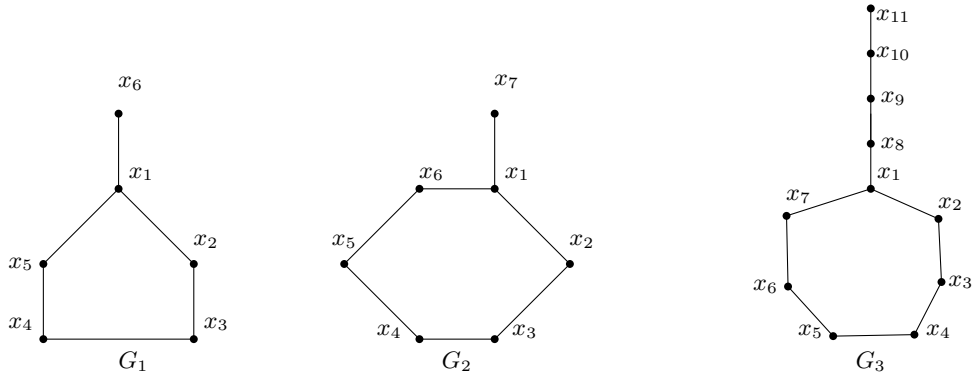
\begin{figure}[H]
	    \centering
     \begin{tikzpicture}[scale=1]
	    \draw  (-7,1)-- (-7,0);
\draw  (-7,0)-- (-8,-1);
\draw  (-8,-1)-- (-8,-2);
\draw  (-8,-2)-- (-6,-2);
\draw  (-6,-2)-- (-6,-1);
\draw  (-6,-1)-- (-7,0);
\draw  (-2,1)-- (-2,0);
\draw  (-3,0)-- (-2,0);
\draw  (-3,0)-- (-4,-1);
\draw  (-4,-1)-- (-3,-2);
\draw  (-3,-2)-- (-2,-2);
\draw  (-2,-2)-- (-1,-1);
\draw  (-1,-1)-- (-2,0);
\draw  (3,0)-- (1.88,-0.36);
\draw  (1.88,-0.36)-- (1.9,-1.3);
\draw  (1.9,-1.3)-- (2.5,-1.96);
\draw  (2.5,-1.96)-- (3.58,-1.94);
\draw  (3.58,-1.94)-- (3.94,-1.24);
\draw  (3.94,-1.24)-- (3.9,-0.4);
\draw  (3.9,-0.4)-- (3,0);
\draw  (3,0)-- (3,1);
\draw  (3,0.6)-- (3,1.2);
\draw  (3,1.2)-- (3,1.8);
\draw  (3,1.8)-- (3,2.4);
\begin{scriptsize}
\fill (-7,0) circle (1.5pt);
\draw (-6.7,0.2) node {$x_1$};
\fill (-8,-1) circle (1.5pt);
\draw (-8.3,-1) node {$x_5$};
\fill (-6,-1) circle (1.5pt);
\draw (-5.84,-0.7) node {$x_2$};
\fill (-8,-2) circle (1.5pt);
\draw (-8.3,-1.8) node {$x_4$};
\fill (-6,-2) circle (1.5pt);
\draw (-5.7,-1.8) node {$x_3$};
\fill (-3,0) circle (1.5pt);
\draw (-2.84,0.2) node {$x_6$};
\fill  (-2,0) circle (1.5pt);
\draw (-1.7,0.2) node {$x_1$};
\fill  (-4,-1) circle (1.5pt);
\draw (-4,-0.7) node {$x_5$};
\fill  (-1,-1) circle (1.5pt);
\draw (-0.84,-0.7) node {$x_2$};
\fill  (-3,-2) circle (1.5pt);
\draw (-3.4,-1.9) node {$x_4$};
\fill  (-2,-2) circle (1.5pt);
\draw (-1.6,-1.9) node {$x_3$};
\fill  (1.88,-0.36) circle (1.5pt);
\draw (2.04,-0.1) node {$x_7$};
\fill  (1.9,-1.3) circle (1.5pt);
\draw (1.6,-1.3) node {$x_6$};
\fill  (2.5,-1.96) circle (1.5pt);
\draw (2.2,-1.9) node {$x_5$};
\fill  (3.58,-1.94) circle (1.5pt);
\draw (3.9,-1.9) node {$x_4$};
\fill  (3.94,-1.24) circle (1.5pt);
\draw (4.2,-1.24) node {$x_3$};
\fill  (3.9,-0.4) circle (1.5pt);
\draw (4.06,-0.2) node {$x_2$};
\fill  (-7,1) circle (1.5pt);
\draw (-6.84,1.4) node {$x_6$};
\fill  (-2,1) circle (1.5pt);
\draw (-1.84,1.43) node {$x_7$};
\fill  (3,0) circle (1.5pt);
\draw (3.25,0.2) node {$x_1$};
\fill  (3,0.6) circle (1.5pt);
\draw (3.3,0.6) node {$x_8$};
\fill  (3,1.2) circle (1.5pt);
\draw (3.3,1.2) node {$x_9$};
\fill  (3,1.8) circle (1.5pt);
\draw (3.3,1.8) node {$x_{10}$};
\fill  (3,2.4) circle (1.5pt);
\draw (3.3,2.3) node {$x_{11}$};
\draw (-6.8,-2.3) node {$G_1$};
\draw (-2.5,-2.3) node {$G_2$};
\draw (3,-2.3) node {$G_3$};
\end{scriptsize}
\end{tikzpicture}
	    \caption{Unicyclic graphs}
	    \label{fig:unicyclic}
	\end{figure}
\end{example}
 

\noindent
\textbf{Conflict of interest statement: } Not applicable.
\\ 
\textbf{Data availability statement: } Not applicable.
	\bibliographystyle{plain}
	\bibliography{reference}
\end{document}